\documentclass[12pt,twoside]{amsart}
\usepackage{amssymb}
\usepackage{amscd}
\usepackage{xypic} 

\title[Theory of non-lc ideal sheaves]
{Theory of non-lc ideal sheaves\\ ---Basic Properties---}
\author{Osamu Fujino} 
\subjclass[2000]{Primary 14E15; Secondary 14E30.}
\date{2009/10/25, version 1.01}
\keywords{log canonical singularities, 
multiplier ideal sheaves, 
inversion of 
adjunction}
\address{Department of Mathematics, Faculty of Science, Kyoto University, 
Kyoto 606-8502 Japan}
\email{fujino@math.kyoto-u.ac.jp}
\newcommand{\Nlc}[0]{{\operatorname{Nlc}}}
\newcommand{\Ker}[0]{{\operatorname{Ker}}}
\newcommand{\Spec}[0]{{\operatorname{Spec}}}
\newcommand{\mult}[0]{{\operatorname{mult}}}
\newcommand{\im}[0]{{\operatorname{Im}}}
\newcommand{\Exc}[0]{{\operatorname{Exc}}}

\newcommand{\Supp}[0]{{\operatorname{Supp}}}
\newtheorem{thm}{Theorem}[section]
\newtheorem{lem}[thm]{Lemma}
\newtheorem{cor}[thm]{Corollary}
\newtheorem{prop}[thm]{Proposition}
\newtheorem*{claim}{Claim}

\theoremstyle{definition}
\newtheorem{ex}[thm]{Example}
\newtheorem{defn}[thm]{Definition}
\newtheorem{rem}[thm]{Remark}
\newtheorem*{ack}{Acknowledgments}       
\newtheorem*{notation}{Notation}         
\newtheorem{que}[thm]{Question}
\newtheorem{step}{Step}
\begin{document}
\bibliographystyle{amsalpha+}

\begin{abstract}
We introduce the notion of non-lc ideal sheaves. 
It is an analogue of the notion of multiplier ideal 
sheaves. 
We establish the restriction theorem, which seems to be 
the most important property of non-lc ideal sheaves. 
\end{abstract}

\maketitle

\section{{Introduction}}\label{intro}

Let $X$ be a smooth complex algebraic 
variety and $B$ an effective $\mathbb R$-divisor on $X$. 
Then we can define the {\em{multiplier ideal sheaf}} 
$\mathcal J(X, B)$. 
By the definition, $(X, B)$ is klt if and only if 
$\mathcal J(X, B)$ is trivial. 
There exist plenty of applications 
of multiplier ideal sheaves. 
See, for example, the excellent book \cite{lazarsfeld}. 
Here, we introduce the notion of 
{\em{non-lc ideal sheaves}}. 
We denote it by $\mathcal J_{NLC}(X, B)$. 
By the construction, the ideal sheaf $\mathcal J_{NLC}(X, B)$ is trivial 
if and only if $(X, B)$ is lc, that is, $\mathcal J_{NLC}(X, B)$ 
defines 
the non-lc locus of the pair $(X, B)$. 
So, we call $\mathcal J_{NLC}(X, B)$ the {\em{non-lc ideal 
sheaf}} associated to $(X, B)$. 
By the definition of 
$\mathcal J_{NLC}(X, B)$ (cf.~Definition \ref{21}), we have 
the following inclusions 
$$
\mathcal J(X, B)\subset \mathcal J_{NLC}(X, B)\subset 
\mathcal J(X, (1-\varepsilon)B)
$$ 
for every $\varepsilon >0$. Although 
the ideal sheaf $\mathcal J(X, (1-\varepsilon)B)$ 
defines the non-lc locus of the pair 
$(X, B)$ for 
$0<\varepsilon \ll 1$, 
$\mathcal J(X, (1-\varepsilon)B)$ does not 
always coincide 
with $\mathcal J_{NLC}(X, B)$. 
It is a very important remark. 
In \cite{fuji-taka}, we will discuss various other ideal sheaves which define 
the non-lc locus of $(X, B)$. 
 
Let $S$ be a smooth 
irreducible divisor on $X$ such that 
$S$ is not contained in the support of $B$. 
We put $B_S=B|_S$. 
The restriction theorem for multiplier ideal 
sheaves, which was obtained by Esnault--Viehweg, 
is one of the key results 
in the theory of 
multiplier ideal sheaves. 
From the analytic point of view, it is a direct consequence of 
the Ohsawa--Takegoshi $L^2$ extension theorem (see \cite{ot}). 
For the details, see \cite{kollar} and \cite{lazarsfeld}. 
Let us recall the restriction theorem here for the reader's convenience. 
\begin{thm}[Restriction Theorem for Multiplier Ideal Sheaves] 
We have an inclusion 
$$
\mathcal J(S, B_S)\subseteq \mathcal J(X, B)|_S. 
$$ 
\end{thm}
The main result of this paper 
is the following restriction theorem for non-lc ideal sheaves. 
For the precise statement, see Theorem \ref{main}. 

\begin{thm}
There is an equality 
$$
\mathcal J_{NLC}(S, B_S)= \mathcal J_{NLC}(X, S+B)|_S. 
$$
In particular, $(S, B_S)$ is lc if and only if $(X, S+B)$ is lc around $S$. 
\end{thm}
Once we obtain this powerful restriction theorem for non-lc ideal 
sheaves, we can translate some results 
for multiplier ideal sheaves 
into new results for non-lc ideal sheaves. 
We will prove, for example,  
subadditivity theorem for 
non-lc ideal sheaves. 
I think that the ideal sheaf $\mathcal J_{NLC}(X, B)$ has already 
appeared implicitly in some papers. 
However, $\mathcal J_{NLC}(X, B)$ was thought to be 
useless because the Kawamata--Viehweg--Nadel vanishing theorem does 
not hold for lc pairs. 
We note that the theory of multiplier ideal sheaves 
heavily depends on the Kawamata--Viehweg--Nadel vanishing 
theorem. 
Fortunately, we have a new cohomological package 
according to Ambro's formulation, which works 
for lc pairs (see \cite[Chapter 2]{fujino-book}). 
By this new package, 
we can walk around freely in the world of lc pairs. 
We will prove vanishing theorem and global generation 
for non-lc ideal sheaves 
as applications. 
I hope that the notion of non-lc ideal sheaves will play important roles 
in various applications. 
In \cite{fuji-f}, we prove the cone and contraction theorem 
for a pair $(X, B)$ where $X$ is a normal variety and 
$B$ is an effective $\mathbb R$-divisor 
on $X$ such that $K_X+B$ is $\mathbb R$-Cartier. 
In that paper, we repeatedly use non-lc ideal sheaves. We note 
that the restriction theorem (cf.~Theorem \ref{main}) is not necessary 
in \cite{fuji-f}. 
We only use the basic properties of non-lc ideal sheaves. 

We summarize the contents of this paper. 
In Section \ref{sec2}, we introduce the notion 
of non-lc ideal sheaves and give 
various examples. 
Then we prove the restriction theorem for non-lc 
ideal sheaves. 
It produces 
the subadditivity theorem for non-lc ideal sheaves, and so on. 
Our proof of the restriction theorem 
is quite different from 
the standard arguments in the theory of multiplier ideal 
sheaves in \cite{lazarsfeld}. 
It also differs from the usual X-method, which was initiated 
by Kawamata and is the most important technique in the 
traditional log minimal model program. 
So, we will explain the proof of the restriction theorem 
very carefully. 
In Section \ref{sec3}, we prove the vanishing 
theorem and the global generation 
for (asymptotic) 
non-lc ideal sheaves. 
Section \ref{sec4} is an appendix, where 
we quickly review Kawakita's inversion of 
adjunction on log canonicity and 
the new cohomological package (cf.~\cite[Chapter 2]{fujino-book}). 

\begin{ack}
I was partially supported by the Grant-in-Aid for Young Scientists 
(A) $\sharp$20684001 from JSPS. I was 
also supported by the Inamori Foundation. 
I am grateful to Professor Hiromichi Takagi for his warm 
encouragement. 
Almost all the results were obtained during my short trip to 
Germany with the book \cite{lazarsfeld} in 
the autumn of 2007. 
I thank the Max-Planck Institute for Mathematics and 
Mathematisches Forschungsinstitut Oberwolfach for 
their hospitality. I am grateful to Vladimir Lazi\'c and 
Masayuki Kawakita for useful 
comments. I also thank Professors 
Mircea Musta\c t\u a 
and Mihnea Popa for 
pointing out mistakes. 
I am grateful to the Mathematical Sciences Research 
Institute for its hospitality. I would like to 
thank Shunsuke Takagi and Karl Schwede for many questions, 
suggestions, and useful comments. 
\end{ack}

\subsection{Notation and Conventions}
We will work over the complex number field $\mathbb C$ throughout 
this paper. 
But we note that by using the Lefschetz principle, we can 
extend everything to the case where 
the base field is 
an algebraically closed field of characteristic 
zero. 
We closely follow the presentation of the 
excellent book \cite{lazarsfeld} 
in order to 
make this paper 
more accessible. 
We will use the following notation freely. 

\begin{notation}
(i) For an $\mathbb R$-Weil divisor 
$D=\sum _{j=1}^r d_j D_j$ such that 
$D_i$ is a prime divisor for every $i$ and 
$D_i\ne D_j$ for $i\ne j$, we define 
the {\em{round-up}} $\ulcorner D\urcorner =\sum _{j=1}^{r} 
\ulcorner d_j\urcorner D_j$ 
(resp.~the {\em{round-down}} $\llcorner D\lrcorner 
=\sum _{j=1}^{r} \llcorner d_j \lrcorner D_j$), 
where for every real number $x$, 
$\ulcorner x\urcorner$ (resp.~$\llcorner x\lrcorner$) is 
the integer defined by $x\leq \ulcorner x\urcorner <x+1$ 
(resp.~$x-1<\llcorner x\lrcorner \leq x$). 
The {\em{fractional part}} $\{D\}$ 
of $D$ denotes $D-\llcorner D\lrcorner$. 
We define 
\begin{align*}&
D^{=1}=\sum _{d_j=1}D_j, \ \ D^{\leq 1}=\sum_{d_j\leq 1}d_j D_j, \\ &
D^{<1}=\sum_{d_j< 1}d_j D_j, \ \ \text{and}\ \ \ 
D^{>1}=\sum_{d_j>1}d_j D_j. 
\end{align*}
We call $D$ a {\em{boundary}} 
$\mathbb R$-divisor if 
$0\leq d_j\leq 1$ 
for every $j$. 
We note that $\sim _{\mathbb Q}$ (resp.~$\sim _{\mathbb R}$) 
denotes the $\mathbb Q$-linear (resp.~$\mathbb R$-linear) equivalence 
of $\mathbb Q$-divisors (resp.~$\mathbb R$-divisors). 

(ii) For a proper birational morphism $f:X\to Y$, 
the {\em{exceptional locus}} $\Exc (f)\subset X$ is the locus where 
$f$ is not an isomorphism. 

(iii) Let $X$ be a normal variety and $B$ an effective $\mathbb R$-divisor 
on $X$ such that $K_X+B$ is $\mathbb R$-Cartier. 
Let $f:Y\to X$ be a resolution such that 
$\Exc (f)\cup f^{-1}_*B$ has a simple normal crossing 
support, where $f^{-1}_*B$ is the strict transform of $B$ on $Y$. 
We write $$K_Y=f^*(K_X+B)+\sum _i a_i E_i$$ and 
$a(E_i, X, B)=a_i$. 
We say that $(X, B)$ is {\em{lc}} (resp.~{\em{klt}}) if and only if 
$a_i\geq -1$ (resp.~$a_i>-1$) for every $i$. 
Note that the {\em{discrepancy}} $a(E, X, B)\in \mathbb R$ can be 
defined for every prime divisor $E$ {\em{over}} $X$. 
By the definition, there exists the largest Zariski open set 
$U$ of $X$ such that $(X, B)$ 
is lc on $U$. 
We put $\Nlc(X, B)=X\setminus U$ and call it 
the {\em{non-lc locus}} of the pair 
$(X, B)$. 
We sometimes simply denote $\Nlc (X, B)$ by $X_{NLC}$. 

(iv) Let $E$ be a prime divisor over $X$. The closure 
of the image of $E$ on $X$ is denoted by 
$c_X(E)$ and 
called the {\em{center}} of $E$ on $X$. 

We use the same notation as in (iii). 
If $a(E, X, B)=-1$ and 
$c_X(E)$ is not contained in $\Nlc(X, B)$, then 
$c_X(E)$ is called 
an {\em{lc center}} of $(X, B)$. We note that 
our definition 
of lc centers is slightly different from the usual one. 
\end{notation}

\section{Non-lc Ideal Sheaves}\label{sec2}

\subsection{Definitions of Non-lc Ideal Sheaves} 

Let us introduce the notion of {\em{non-lc ideal sheaves}}. 
\begin{defn}[Non-lc ideal sheaf]\label{21}
Let $X$ be a normal variety and $\Delta$ 
an $\mathbb R$-divisor on $X$ such that 
$K_X+\Delta$ is $\mathbb R$-Cartier. 
Let $f:Y\to X$ be a resolution with $K_Y+\Delta _Y=f^*(K_X+\Delta)$ 
such that 
$\Supp \Delta _Y$ is simple normal crossing. 
Then we put 
$$\mathcal J_{NLC}(X, \Delta)=f_*\mathcal O_Y(\ulcorner -(\Delta_Y^{<1})\urcorner-
\llcorner \Delta_Y^{>1}\lrcorner)=f_*\mathcal O_Y(-\llcorner \Delta_Y\lrcorner 
+\Delta^{=1}_Y)$$ and 
call it the {\em{non-lc ideal sheaf associated to $(X, \Delta)$}}. 
\end{defn}

The name comes from the following obvious lemma. 
See also Proposition \ref{2626}. 

\begin{lem}
Let $X$ be a normal variety and $\Delta$ an effective 
$\mathbb R$-divisor 
such that $K_X+\Delta$ is $\mathbb R$-Cartier. 
Then $(X, \Delta)$ is lc if and only if $\mathcal J_{NLC}(X, \Delta)=\mathcal O_X$. 
\end{lem}

\begin{rem}
In the same notation as in Definition \ref{21}, 
we put 
$$
\mathcal J(X, \Delta)=f_*\mathcal O_Y(-\llcorner \Delta_Y\lrcorner)
=f_*\mathcal O_Y(K_Y-\llcorner f^*(K_X+\Delta)\lrcorner). 
$$ 
It is nothing but the well-known {\em{multiplier ideal sheaf}}. 
It is obvious that $\mathcal J(X, \Delta)\subseteq \mathcal J_{NLC}(X, \Delta)$. 
\end{rem}

\begin{que}
Let $X$ be a smooth algebraic variety and $\Delta$ an effective 
$\mathbb R$-divisor on $X$. 
Are there any analytic interpretations of $\mathcal J_{NLC}(X, \Delta)^{an}$? 
Are there any approaches to $\mathcal J_{NLC}(X, \Delta)$ 
from the theory of tight closure? 
\end{que}

\begin{defn}[Non-lc ideal sheaf associated to an ideal sheaf]
Let $X$ be a normal variety and $\Delta$ an 
$\mathbb R$-divisor on $X$ such that 
$K_X+\Delta$ is $\mathbb R$-Cartier. 
Let $\mathfrak a\subseteq \mathcal O_X$ be a non-zero 
ideal sheaf on $X$ and $c$ a real 
number. 
Let $f:Y\to X$ be a resolution such that 
$K_Y+\Delta_Y=f^*(K_X+\Delta)$ and 
that 
$f^{-1}\mathfrak a =\mathcal O_Y(-F)$, 
where 
$\Supp (\Delta_Y+F)$ has a simple normal crossing support. 
We put 
$$
\mathcal J_{NLC}((X, \Delta); \mathfrak a^c)=f_*\mathcal 
O_Y(\ulcorner -((\Delta _Y+cF)^{<1})\urcorner 
-\llcorner (\Delta_Y+cF)^{>1}\lrcorner).  
$$ 
We sometimes write $\mathcal J_{NLC}((X, \Delta); c\cdot \mathfrak a)
=\mathcal J_{NLC}((X, \Delta); \mathfrak a^c)$. 
\end{defn}

\begin{prop}\label{2626}
The ideal sheaves $\mathcal J_{NLC}(X, \Delta)$ and $\mathcal J_{NLC}
((X, \Delta); \mathfrak a^c)$ are well-defined, that is, 
they are independent of the resolution $f:Y\to X$. 
If $\Delta$ is effective and $c>0$, 
then $\mathcal J_{NLC}(X, \Delta)\subseteq \mathcal O_X$ 
and $\mathcal J_{NLC}((X, \Delta); \mathfrak a^c)\subseteq \mathcal O_X$. 
\end{prop}

This proposition follows from the next fundamental lemma. 

\begin{lem}\label{taisetsu}
Let $f:Z\to Y$ be a proper birational morphism between 
smooth varieties and $B_Y$ an 
$\mathbb  R$-divisor on $Y$ such 
that $\Supp B_Y$ is simple normal crossing. 
Assume that $K_Z+B_Z=f^*(K_Y+B_Y)$ and 
that $\Supp B_Z$ is simple normal crossing. 
Then we have $$f_*\mathcal O_Z(\ulcorner -(B^{<1}_Z)\urcorner 
-\llcorner B^{>1}_Z\lrcorner)\simeq 
\mathcal O_Y(\ulcorner -(B^{<1}_Y)\urcorner 
-\llcorner B^{>1}_Y\lrcorner). $$ 
\end{lem}
\begin{proof}
By $K_Z+B_Z=f^*(K_Y+B_Y)$, we 
obtain 
\begin{align*}
K_Z=&f^*(K_Y+B^{=1}_Y+\{B_Y\})\\&+f^*
(\llcorner B^{<1}_Y\lrcorner+\llcorner B^{>1}_Y\lrcorner)
-(\llcorner B^{<1}_Z\lrcorner+\llcorner B^{>1}_Z\lrcorner)
-B^{=1}_Z-\{B_Z\}.
\end{align*} 
If $a(\nu, Y, B^{=1}_Y+\{B_Y\})=-1$ for a prime divisor 
$\nu$ over $Y$, then 
we can check that $a(\nu, Y, B_Y)=-1$ by using 
\cite[Lemma 2.45]{km}. 
Since $f^*
(\llcorner B^{<1}_Y\lrcorner+\llcorner B^{>1}_Y\lrcorner)
-(\llcorner B^{<1}_Z\lrcorner+\llcorner B^{>1}_Z\lrcorner)$ is 
Cartier, we can easily see that 
$$f^*(\llcorner B^{<1}_Y\lrcorner+\llcorner B^{>1}_Y\lrcorner)
=\llcorner B^{<1}_Z\lrcorner+\llcorner B^{>1}_Z\lrcorner+E, $$ 
where $E$ is an effective $f$-exceptional divisor. 
Thus, we obtain 
$$f_*\mathcal O_Z(\ulcorner -(B^{<1}_Z)\urcorner 
-\llcorner B^{>1}_Z\lrcorner)\simeq 
\mathcal O_Y(\ulcorner -(B^{<1}_Y)\urcorner 
-\llcorner B^{>1}_Y\lrcorner).$$  
We finish the proof. 
\end{proof}
Although the following 
lemma is not indispensable for the proof of 
the main theorem, it may be useful. 
The proof is quite nontrivial. 

\begin{lem}\label{tai2}
We use the same notation and assumption 
as in {\em{Lemma \ref{taisetsu}}}.  
Let $S$ be a simple normal crossing divisor on $Y$ such 
that $S\subset \Supp B^{=1}_Y$. Let $T$ be the union of the 
irreducible 
components of $B^{=1}_Z$ that are mapped into $S$ by $f$. 
Assume that 
$\Supp f^{-1}_*B_Y\cup \Exc (f)$ is simple normal crossing on $Z$. 
Then we have 
$$f_*\mathcal O_T(\ulcorner -(B^{<1}_T)\urcorner 
-\llcorner B^{>1}_T\lrcorner)\simeq 
\mathcal O_S(\ulcorner -(B^{<1}_S)\urcorner 
-\llcorner B^{>1}_S\lrcorner),$$ 
where 
$(K_Z+B_Z)|_T=K_T+B_T$ and $(K_Y+B_Y)|_S=K_S+B_S$. 
\end{lem}
\begin{proof} 
We use the same notation as in the proof of Lemma \ref{taisetsu}. 
We consider 
\begin{align*}
0&\to \mathcal O_Z(\ulcorner -(B^{<1}_Z)\urcorner-
\llcorner B^{>1}_Z\lrcorner-T)\\ &\to 
\mathcal O_Z(\ulcorner -(B^{<1}_Z)\urcorner-
\llcorner B^{>1}_Z\lrcorner)
\to \mathcal O_T(\ulcorner -(B^{<1}_T)\urcorner-
\llcorner B^{>1}_T\lrcorner)\to 0. 
\end{align*} 
Since $T=f^*S-F$, where $F$ is an effective $f$-exceptional 
divisor, we can easily see that 
$$f_*\mathcal O_Z(\ulcorner -(B^{<1}_Z)\urcorner 
-\llcorner B^{>1}_Z\lrcorner-T)\simeq 
\mathcal O_Y(\ulcorner -(B^{<1}_Y)\urcorner 
-\llcorner B^{>1}_Y\lrcorner-S).$$  
We note that 
\begin{align*} 
&(\ulcorner -(B^{<1}_Z)\urcorner 
-\llcorner B^{>1}_Z\lrcorner-T)-(K_Z+\{B_Z\}+(B^{=1}_Z-T))
\\ &=-f^*(K_Y+B_Y). 
\end{align*} 
Therefore, every local section 
of $R^1f_*\mathcal O_Z(\ulcorner -(B^{<1}_Z)\urcorner 
-\llcorner B^{>1}_Z\lrcorner-T)$ contains in its support the $f$-image 
of some stratum of $(Z, \{B_Z\}+B^{=1}_Z-T)$ by 
Theorem \ref{ap1} (1). 
\begin{claim}
No strata of $(Z, \{B_Z\}+B^{=1}_Z-T)$ are 
mapped into $S$ by $f$. 
\end{claim}
\begin{proof}[Proof of Claim]
Assume that there is a stratum $C$ of $(Z, \{B_Z\}+B^{=1}_Z-T)$ such that 
$f(C)\subset S$. Note that 
$\Supp f^*S\subset \Supp f^{-1}_*B_Y\cup \Exc (f)$ and 
$\Supp B^{=1}_Z\subset \Supp f^{-1}_*B_Y\cup \Exc (f)$. 
Since $C$ is also a stratum of $(Z, B^{=1}_Z)$ and 
$C\subset \Supp f^*S$, 
there exists an irreducible component $G$ of $B^{=1}_Z$ such that 
$C\subset G\subset \Supp f^*S$. 
Therefore, by the definition of $T$, $G$ is an 
irreducible component of $T$ because $f(G)\subset S$ and $G$ is an 
irreducible component of $B^{=1}_Z$. So, 
$C$ is not a stratum of $(Z, \{B_Z\}+B^{=1}_Z-T)$. It is 
a contradiction. 
\end{proof}
On the other hand, $f(T)\subset S$. Therefore, 
$$f_*\mathcal O_T(\ulcorner -(B^{<1}_T)\urcorner
-\llcorner B^{>1}_T\lrcorner)
\to R^1f_*\mathcal O_Z(\ulcorner -(B^{<1}_Z)\urcorner
-\llcorner B^{>1}_Z\lrcorner-T)$$ is a zero-map by the above 
claim. 
Thus, we obtain  
$$f_*\mathcal O_T(\ulcorner -(B^{<1}_T)\urcorner 
-\llcorner B^{>1}_T\lrcorner)\simeq 
\mathcal O_S(\ulcorner -(B^{<1}_S)\urcorner 
-\llcorner B^{>1}_S\lrcorner).$$ 
We finish the proof.  
\end{proof}

\begin{rem}\label{28rem} 
Let $X$ be an $n$-dimensional normal variety and $\Delta$ 
an $\mathbb R$-divisor on $X$ such that 
$K_X+\Delta$ is $\mathbb R$-Cartier. 
Let $f:Y\to X$ be a resolution 
with 
$K_Y+\Delta_Y=f^*(K_X+\Delta)$ 
such that $\Supp \Delta_Y$ is simple normal crossing. 
We put $A=\ulcorner -(\Delta^{<1}_Y)\urcorner$, $N=\llcorner 
\Delta^{>1}_Y\lrcorner$, and 
$W=\Delta^{=1}_Y$. 
Since $R^if_*\mathcal O_Y(A-N-W)=0$ for $i>0$ by the 
Kawamata--Viehweg vanishing theorem, 
we have 
$$
0\to \mathcal J(X, \Delta)\to \mathcal J_{NLC}(X, \Delta)\to f_*\mathcal O_W
(A|_W-N|_W)\to 0, 
$$ 
and 
$$
R^if_*\mathcal O_Y(A-N)\simeq R^if_*\mathcal O_W(A|_W-N|_W)
$$ 
for every $i>0$. 
In general, $R^if_*\mathcal O_Y(A-N)\ne 0$ for 
$1\leq i\leq n-1$. 

From now on, we assume that $\Delta$ is effective. We put 
$F=W-E$, where $E$ is the union of irreducible components of $W$ which 
are mapped to $\Nlc(X, \Delta)$. 
Then we have 
$$
f_*\mathcal O_Y(A-N-E)= f_*\mathcal O_Y(A-N)=\mathcal J_{NLC}
(X, \Delta). 
$$ 
Applying $f_*$ to the following short exact sequence 
$$
0\to \mathcal O_Y(A-N-W)\to \mathcal O_Y(A-N-E)\to \mathcal O_F(A|_F-N|_F
-E|_F)\to 0, 
$$
we obtain that 
$$
f_*\mathcal O_F(A|_F-N|_F-E|_F)= f_*\mathcal O_W(A|_W-N|_W). 
$$ 
In particular, $\mathcal J(X, \Delta)= \mathcal J_{NLC}(X, \Delta)$ 
if and only if $(X, \Delta)$ has no lc centers. 
\end{rem}

\subsection{Examples of Non-lc Ideal Sheaves}

Here, we explain some elementary examples. 

\begin{ex}
Let $X$ be an $n$-dimensional 
smooth variety. 
Let $P\in X$ be a closed point and $\mathfrak m=\mathfrak 
m_P$ the associated maximal 
ideal. 
Let $f:Y\to X$ be the blow-up 
at $P$. 
Then $f^{-1}\mathfrak m =\mathcal O_Y(-E)$, where 
$E$ is the exceptional divisor of $f$. 
If $c>n$, then 
\begin{align*}
\mathcal J_{NLC}(X; c\cdot \mathfrak m)=f_*\mathcal 
O_Y(((n-1)-\llcorner c\lrcorner )E)=
\mathcal J(X; c\cdot \mathfrak m)= \mathfrak 
m^{\llcorner c\lrcorner -(n-1)}. 
\end{align*}
If $c<n$, 
then \begin{align*}
\mathcal J_{NLC}(X; c\cdot \mathfrak m)= 
f_*\mathcal O_Y(((n-1)-\llcorner c\lrcorner )E)= \mathcal J(X; 
c\cdot \mathfrak m)= \mathcal O_X. 
\end{align*}
When $c=n$, we note that 
\begin{align*}
\mathcal J_{NLC}(X; c\cdot \mathfrak 
m)=f_*\mathcal O_Y\simeq \mathcal O_X
\supsetneq \mathcal J(X; c\cdot \mathfrak 
m)=f_*\mathcal O_Y(-E)=\mathfrak 
m. 
\end{align*} 
\end{ex}

\begin{ex}
Let $X$ be a smooth variety and $D$ 
a smooth divisor on $X$. 
Then $\mathcal J_{NLC}(X, D)=\mathcal O_X$. 
However, $$\mathcal J_{NLC}(X, (1+\varepsilon)D)=
\mathcal O_X(-D)$$ for every $0<\varepsilon \ll 1$. 
On the other hand, $$\mathcal J(X, D)=\mathcal J(X, (1+\varepsilon)D)
= \mathcal O_X(-D)$$ for every $0<\varepsilon \ll 1$. 
\end{ex}

We note the following lemma on the {\em{jumping numbers}}, 
whose proof is obvious by the definitions (cf.~\cite[Lemma 9.3.21, Definition 9.3.22]
{lazarsfeld}). 

\begin{lem}[Jumping numbers] 
Let $X$ be a smooth variety and $D$ an effective $\mathbb Q$-divisor 
$($resp.~$\mathbb R$-divisor$)$ on $X$. Let 
$x\in X$ be a fixed point contained in the support of $D$. 
Then there is an increasing sequence 
$$
0<\xi_0(D; x) <\xi _1(D; x)<\xi _2(D; x)<\cdots 
$$ 
of rational $($resp.~real$)$ numbers $\xi_i =\xi_i (D; x)$ characterized 
by the properties that 
$$
\mathcal J(X, c\cdot D)_x=\mathcal J(X, \xi_i\cdot D)_x \ \ \text{for}\ \  c\in 
[\xi_i, \xi _{i+1}), 
$$ 
while 
$\mathcal J(X, \xi _{i+1}\cdot D)_x\subsetneq \mathcal J(X, \xi_i\cdot D)_x$ for 
every $i$. 
The rational $($resp.~real$)$ numbers $\xi_i(D; x)$ are called the {\em{jumping 
numbers}} of $D$ at $x$. 
We can check the properties that 
$$
\mathcal J_{NLC}(X, c\cdot D)_x=\mathcal J_{NLC}(X, d\cdot D)_x \ \ \text{for} \ \ 
c, d \in 
(\xi_i, \xi _{i+1}), 
$$ 
while 
$\mathcal J_{NLC}(X, \xi _{i+1}\cdot D)_x\subsetneq \mathcal J_{NLC}
(X, \xi_i\cdot D)_x$ for 
every $i$. 
Moreover, 
$\mathcal J_{NLC}(X, c\cdot D)_x=\mathcal J(X, c\cdot D)_x$ for 
$c\in (\xi_i, \xi_{i+1})$ by {\em{Remark \ref{28rem}}}. 
\end{lem}

\begin{ex}
Let $X=\mathbb C^2=\Spec \mathbb C[z_1, z_2]$ and 
$D=(z_1=0)+(z_2=0)+(z_1=z_2)$. 
Then we can directly check that 
$$\mathcal J_{NLC}(X, D)=\mathfrak m^2$$ and 
$$\mathcal J_{NLC}(X, (1-\varepsilon)D)=\mathcal J(X, 
(1-\varepsilon)D)=\mathfrak m$$ for 
$0<\varepsilon \ll 1$, where 
$\mathfrak m$ is the maximal ideal associated to $0\in 
\mathbb C^2$. 
On the other hand, $$\mathcal J_{NLC}(X, (1+\varepsilon)D)
=\mathcal J(X, (1+\varepsilon)D)\subsetneq 
\mathcal J_{NLC}(X, D)$$ for $0<\varepsilon \ll 1$ because 
$D\subset \Nlc(X, (1+\varepsilon)D)$. 
Note that $$\mathcal J(X, D)=\mathcal J(X, (1+\varepsilon)D)
\subsetneq \mathcal J_{NLC}(X, D)$$ for $0<\varepsilon 
\ll 1$.   
\end{ex}

\subsection{Main Theorem:~Restriction Theorem} 
The following theorem is the main theorem of this paper. 

\begin{thm}[Restriction Theorem]\label{main}
Let $X$ be a normal variety and $S+B$ an effective $\mathbb R$-divisor 
on $X$ such that $S$ is reduced and normal and that $S$ and $B$ have no 
common irreducible components. 
Assume that $K_X+S+B$ is $\mathbb R$-Cartier. 
Let $B_S$ be the different on $S$ such that 
$K_S+B_S=(K_X+S+B)|_S$. 
Then we obtain 
$$
\mathcal J_{NLC}(S, B_S)= \mathcal J_{NLC}(X, S+B)|_S. 
$$ 
In particular, $(S, B_S)$ is log canonical if and only if 
$(X, S+B)$ is log canonical around $S$. 
\end{thm}

\begin{rem}
The notion of {\em{different}} 
was introduced by Shokurov in \cite[\S 3]{shokurov}. 
For the definition and the basic properties, 
see, for example, \cite[9.2.1]{ambro} and 
\cite[Section 14]{fuji-f}. 
\end{rem}

Before we start the proof of Theorem \ref{main}, let us 
see an easy example. 

\begin{ex}
Let $X=\mathbb C^2=\Spec \mathbb C[x, y]$, 
$S=(x=0)$, 
and $B=(y^2=x^3)$. We put $B_S=B|_S$. 
Then we have $K_S+B_S=(K_X+S+B)|_S$. 
By direct calculations, we obtain  
$$
\mathcal J_{NLC}(S, B_S)=\mathfrak m^2, \ \ \ \ 
\mathcal J_{NLC}(X, S+B)=\mathfrak n^2, 
$$
where $\mathfrak m$ (resp.~$\mathfrak n$) is the maximal 
ideal corresponding to $0\in S$ (resp.~$(0, 0)\in X$). 
Of course, we have 
$$
\mathcal J_{NLC}(S, B_S)=\mathcal J_{NLC}(X, S+B)|_S. 
$$ 
\end{ex}

Let us start the proof of Theorem \ref{main}. 

\begin{proof}[Proof of {\em{Theorem \ref{main}}}]
We take a resolution 
$f:Y\to X$ with the following 
properties. 

\begin{itemize}
\item[(i)]$\Exc (f)$ is a simple normal crossing 
divisor on $Y$. 
\item[(ii)] $f^{-1}X_{NLC}$ is a simple normal crossing 
divisor on $Y$, where $X_{NLC}=\Nlc(X, S+B)$. 
\item[(iii)] $f^{-1}S$ is a simple normal crossing 
divisor on $Y$. 
\item[(iv)] $f^{-1}(X_{NLC}\cap S)$ is a simple normal crossing 
divisor on $Y$. 
\item[(v)] $\Exc (f)\cup f^{-1}X_{NLC}\cup f^{-1}_*B\cup f^{-1}S$ is 
a divisor with a simple normal crossing support. 
\end{itemize}
We put $K_Y+B_Y=f^*(K_X+S+B)$. 
Then $\Supp B_Y$ is simple normal crossing 
by (i) and (v). 
Let $S_Y$ be the strict transform of $S$ on $Y$. 
Let $T$ be the union of the components of $B^{=1}_Y-S_Y$ that are 
mapped into $S$ by $f$. 
We can decompose $T=T_1+T_2$ as follows. 

\begin{itemize}
\item[(a)] Any irreducible component of $T_2$ is 
mapped into $X_{NLC}$ by $f$. 
\item[(b)] Any irreducible component of $T_1$ is not mapped 
into $X_{NLC}$ by $f$. 
\end{itemize}
By (ii) and (v), any stratum of $T_1$ is not mapped into 
$X_{NLC}$ by $f$. 

We put $A=\ulcorner -(B^{<1}_Y)\urcorner$ and $N=\llcorner B^{>1}_Y\lrcorner$. 
Then $A$ is an effective $f$-exceptional 
divisor. 
Moreover, $A|_{S_{Y}}$ is exceptional with respect to 
$f:S_Y\to S$. 
Then we have 
$$
\mathcal J_{NLC}(X, S+B)=f_*\mathcal O_Y(A-N)  
$$ 
and $$
\mathcal J_{NLC}(S, B_S)=f_*\mathcal O_{S_Y}(A-N). 
$$
Here, we used 
$$
K_{S_Y}+(B_Y-S_Y)|_{S_Y}=f^*(K_S+B_S). 
$$ 
It follows from 
$K_Y+B_Y=f^*(K_X+S+B)$ by adjunction. 

\begin{step}\label{step1}
We consider the following short exact sequence 
$$
0\to \mathcal O_Y(A-N-(S_Y+T))\to \mathcal O_Y(A-N)\to 
\mathcal O_{S_Y+T}(A-N)\to 0. 
$$ 
Applying $R^if_*$, we obtain that 
\begin{align*}
0&\to f_*\mathcal O_Y(A-N-(S_Y+T))\to f_*\mathcal O_Y(A-N)\\ 
&\to 
f_*\mathcal O_{S_Y+T}(A-N)\to R^1f_*\mathcal O_Y(A-N-(S_Y+T))\to 
\cdots. 
\end{align*}
We note that 
\begin{align*}
&A-N-(S_Y+T)-(K_Y+\{B_Y\}+(B^{=1}_Y-S_Y-T))
\\ &=-f^*(K_X+S+B)
\end{align*} 
and that any stratum of $B^{=1}_Y-S_Y-T$ is not mapped into 
$S$ by $f$ (see the conditions (iii) and (v)). 
Therefore, the support of every non-zero local section of 
$R^1f_*\mathcal O_Y(A-N-(S_Y+T))$ can not be contained in $S$ 
by Theorem \ref{ap1} (1). 
Thus,  
the connecting homomorphism 
$$
f_*\mathcal O_{S_Y+T}(A-N)\to R^1f_*\mathcal O_Y(A-N-(S_Y+T))
$$ 
is a zero-map. Thus, we obtain 
$$
0\to J\to \mathcal J_{NLC}(X, S+B)\to I\to 0, 
$$
where $I:=f_*\mathcal O_{S_Y+T}(A-N)$ and $J:=f_*\mathcal O_Y(A-N-(S_Y+T))$. 
We note that 
the ideal sheaf $J=f_*\mathcal O_Y(A-N-(S_Y+T))\subset \mathcal O_X$ 
defines a scheme structure on $S'=S\cup X_{NLC}$. 
We will check that 
$I\subset \mathcal O_S$ and 
$I= \mathcal J_{NLC}(X, S+B)|_S$ by 
$f(S_Y+T)=S$ and the following commutative diagrams: 
$$
\xymatrix{
0\ar[r] & J\ar[d]^{=}\ar[r]&\mathcal J_{NLC}(X, S+B)\ar[d]\ar[r]&I \ar[d]
\ar[r] &0\\
0\ar[r] & J
\ar[r]&\mathcal O_X\ar[r]& \mathcal O_{S'}\ar[r] & 0,
}
$$
and 
$$
\xymatrix{
0\ar[r] & J\ar[d]\ar[r]&\mathcal O_X\ar[d]^{=}\ar[r]&\mathcal O_{S'} \ar[d]^{\alpha}
\ar[r] &0\\
0\ar[r] & \mathcal O_X(-S)
\ar[r]&\mathcal O_X\ar[r]& \mathcal O_{S}\ar[r] & 0.
}
$$
It is sufficient to 
prove $\Ker \alpha \cap I=\{0\}$, 
where $\alpha:\mathcal O_{S'}\to \mathcal O_S$. 
We note that 
$I=\mathcal J_{NLC}(X, S+B)/J$ and 
$\Ker \alpha =\mathcal O_X(-S)/J$. 
It is easy to see that $$\mathcal J_{NLC}(X, S+B)\cap 
\mathcal O_X(-S)\subset J$$ since $f(S_Y+T)=S$. 
Thus, $\Ker \alpha \cap I=\{0\}$. This means that 
$I\subset \mathcal O_S$ and $I=\mathcal J_{NLC}(X, S+B)|_S$. 

Therefore, it is enough to 
prove $I= \mathcal J_{NLC}(S, B_S)$. 
\end{step}

\begin{step}\label{step2}
In this step, we will prove the
following natural inclusion
$$
f_*\mathcal O_{S_Y+T_1}(A-N-T_2)
\subset
f_*\mathcal O_{S_Y+T}(A-N)=I
$$
is an isomorphism.
We consider the short exact sequence
\begin{align*}
0&\to \mathcal O_Y(A-N-(S_Y+T))\to
\mathcal O_Y(A-N-T_2)
\\ &\to
\mathcal O_{S_Y+T_1}(A-N-T_2)\to 0
\end{align*}
Applying $R^if_*$, we obtain that
\begin{align*}
0& \to J\to f_*\mathcal O_Y(A-N-T_2)
\to f_*\mathcal O_{S_Y+T_1}(A-N-T_2)
\\ &\overset{\delta}\to R^1f_*\mathcal O_Y(A-N-(S_Y+T))
\to \cdots.
\end{align*}
The connecting homomorphism $\delta$ is zero by the completely
same reason as in Step \ref{step1}.
Therefore, we obtain the following commutative diagram.
$$
\xymatrix{
0\ar[r] & J\ar[d]^{=}\ar[r]&f_*\mathcal O_Y(A-N-T_2)\ar[d]^{\beta}\ar[r]&
f_*\mathcal O_{S_Y+T_1}(A-N-T_2)  \ar[d]
\ar[r] &0\\
0\ar[r] & J
\ar[r]&f_*\mathcal O_Y(A-N)\ar[r]& I\ar[r] & 0.
}
$$
The homomorphism $\beta$ is an isomorphism
since $f(T_2)\subset f(N)=X_{NLC}$.
Therefore, we obtain
$$f_*\mathcal O_{S_Y+T_1}(A-N-T_2)=I\subset \mathcal O_S. $$
\end{step}

\begin{step}\label{step3}
The inclusion 
$$
f_*\mathcal O_{S_Y}(A-N-T_2)\subset 
f_*\mathcal O_{S_Y}(A-N)=\mathcal J_{NLC}(S, B_S)\subset \mathcal O_S
$$ 
is obvious. 
By Kawakita's inversion of adjunction on log canonicity 
(cf.~Corollary \ref{renketsu}), 
we obtain the opposite inclusion 
$$
f_*\mathcal O_{S_Y}(A-N)\subset f_*\mathcal O_{S_Y}(A-N-T_2). 
$$ 
Therefore, we obtain 
$$
f_*\mathcal O_{S_Y}(A-N-T_2)=f_*
\mathcal O_{S_Y}(A-N)=\mathcal J_{NLC}
(S, B_S). 
$$
\end{step}

\begin{step}
We consider the following short exact sequence 
\begin{align*}
0&\to \mathcal O_{T_1}(A-N-S_Y-T_2)
\to 
\mathcal O_{S_Y+T_1}(A-N-T_2)
\\ &\to 
\mathcal O_{S_Y}(A-N-T_2)
\to 0. 
\end{align*}
We note that 
$$
f_*\mathcal O_{S_Y+T_1}(A-N-T_2)=I\subset \mathcal O_S  
$$
by Step \ref{step2} and 
$$
f_*\mathcal O_{S_Y}(A-N-T_2)=\mathcal J_{NLC}
(S, B_S) 
$$ by Step \ref{step3}. 
By taking $R^if_*$, 
we obtain 
that 
$$
0\to I\to \mathcal J_{NLC}(S, B_S)\to R^1f_*\mathcal O_{T_1}
(A-N-S_Y-T_2)\to \cdots. 
$$
Here, we used the fact that 
$$f_*\mathcal O_{T_1}(A-N-S_Y-T_2)=0. $$ 
Note that no irreducible components of $S$ are dominated by 
$T_1$. 

Since $\mathcal J_{NLC}(S, B_S)\subset \mathcal O_S$, 
we obtain $$\mathcal J_{NLC}(S, B_S)/I\subset \mathcal O_S/I.$$ 
Since 
\begin{align*}
& A-N- (S_Y+T_2)-(K_Y+T_1+\{B_Y\}+(B^{=1}_Y-S_Y-T))\\
&=-f^*(K_X+S+B), 
\end{align*}
we have 
\begin{align*}
&(A-N-(S_Y+T_2))|_{T_1}-(K_{T_1}+(\{B_Y\}+B^{=1}_Y-S_Y-T)|_{T_1})\\ 
& \sim_{\mathbb R}-f^*(K_X+S+B)|_{f(T_1)}. 
\end{align*}
Therefore, the support of every non-zero local section 
of $R^1f_*\mathcal O_{T_1}(A-N-S_Y-T_2)$ can not be 
contained in $$\Supp (\mathcal O_S/I)\subset 
\Supp (\mathcal O_{S'}/I)=\Supp (\mathcal O_X/\mathcal J_{NLC}(X, S+B))=
X_{NLC}$$ by Theorem \ref{ap1} (1). 
We note that any stratum of $$(T_1, (\{B_Y\}+B^{=1}_Y-S_Y-T)|_{T_1})$$ 
is not mapped into $X_{NLC}$ by $f$ (see the conditions (iv) and (v)). 
Thus, we obtain $I= \mathcal J_{NLC}(S, B_S)$. 
\end{step}
We finish the proof of the main theorem. 
\end{proof}

In some applications, the following 
corollaries may play important roles. 

\begin{cor}\label{c-m}
We use the notation in the proof of {\em{Theorem \ref{main}}}. 
We have the following 
equalities. 
\begin{align*}
\mathcal J_{NLC}(S, B_S)&= 
f_*\mathcal O_{S_Y}(A-N)\\&= 
f_*\mathcal O_{S_Y+T}(A-N)= 
f_*\mathcal O_{S_Y+T_1}(A-N-T_2). 
\end{align*}
\end{cor}

\begin{cor}\label{29} 
We use the notation in the proof of {\em{Theorem \ref{main}}}. 
We obtained the following short exact sequence{\em{:}} 
$$
0\to J\to \mathcal J_{NLC}(X, S+B)\to \mathcal J_{NLC}(S, B_S)\to 0. 
$$
Let $\pi:X\to V$ be a projective morphism onto an algebraic variety $V$ and 
$L$ a Cartier divisor on $X$ such that 
$L-(K_X+S+B)$ is $\pi$-ample. 
Then $$R^i\pi_*(J\otimes \mathcal O_X(L))=0$$ for all $i>0$. 
In particular, $$R^i\pi_*(\mathcal J_{NLC}(X, S+B)\otimes \mathcal O_X(L))\to 
R^i\pi_*(\mathcal J_{NLC}(S, B_S)\otimes \mathcal O_S(L))$$ is 
surjective for $i=0$ and is an isomorphism for every $i\geq 1$.  
As a corollary, we obtain 
$$
\pi_*(\mathcal J_{NLC}(S, B_S)\otimes \mathcal O_S(L))
\subset \im (\pi_*\mathcal O_X(L)\to \pi_*\mathcal O_S(L)). 
$$
\end{cor}
\begin{proof}
Note that we have 
\begin{align*}
&f^*L+A-N-(S_Y+T)-(K_Y+B^{=1}_Y+\{B_Y\}-(S_Y+T))
\\&=f^*(L-(K_X+S+B)). 
\end{align*}
Therefore, $R^i\pi_*(f_*\mathcal O_Y(f^*L+A-N-(S_Y+T)))=0$ for $i>0$ 
by Theorem \ref{ap1} (2). 
Thus, $R^i\pi_*(J\otimes \mathcal O_X(L))=0$ for all $i>0$ because 
$J=f_*\mathcal O_Y(A-N-(S_Y+T))$. 
\end{proof}

\begin{rem}
In Corollary \ref{29}, the ideal $J$ is independent 
of the resolution $f:Y\to X$ by Lemma \ref{tai2}. 
\end{rem}

\begin{rem}
In Corollary \ref{29}, 
we can weaken the assumption 
that $L-(K_X+S+B)$ is $\pi$-ample as follows. 
The $\mathbb R$-Cartier $\mathbb R$-divisor 
$D=L-(K_X+S+B)$ is $\pi$-nef and 
$\pi$-big and $D|_C$ is $\pi$-big 
for every lc center $C$ that is 
not contained in $S$. 
See the proof of Theorem \ref{vani2} below. 
\end{rem}

\subsection{Direct Consequences of Restriction Theorem} 
Let us collect some direct consequences of the restriction theorem. 

\begin{prop}
Let $X$ be a smooth variety, let $D$ be an effective 
$\mathbb R$-divisor on $X$, 
and let $H\subset X$ be a smooth 
irreducible divisor that does not appear in the support of $D$. 
Then 
$$
\mathcal J_{NLC}(H, D|_H)= \mathcal J_{NLC}(X, H+D)|_H
\subseteq \mathcal J_{NLC}(X, D)|_H. 
$$
\end{prop}
\begin{proof}
It is obvious. 
\end{proof}

\begin{cor}
Let $|V|$ be a free linear system, and let $H\in |V|$ be a general 
divisor. 
Then we have 
$$
\mathcal J_{NLC}(H, D|_H)=\mathcal J_{NLC}(X, D)|_H 
$$ 
because $\mathcal J_{NLC}(X, D)=\mathcal J_{NLC}(X, H+D)$. 
\end{cor}
\begin{proof}
It is obvious. 
\end{proof}

\begin{cor}
Let $D$ be an effective $\mathbb R$-divisor on the smooth variety $X$, and 
let $Y\subset X$ be a smooth subvariety that is not contained 
in the support of $D$. 
Then 
$$
\mathcal J_{NLC}(Y, D_Y)\subseteq \mathcal J_{NLC}(X, D)|_Y, 
$$
where $D_Y=D|_Y$. 
\end{cor}
\begin{proof}
It is obvious. See, for example, the proof of \cite[Corollary 9.5.6]
{lazarsfeld}. 
\end{proof}

\begin{cor}
Let $f:Y\to X$ be a morphism of smooth 
irreducible varieties, and let $D$ be an effective $\mathbb R$-divisor 
on $X$. 
Assume that the support of $D$ does not contain $f(Y)$. Then 
one has an inclusion 
$$
\mathcal J_{NLC}(Y, f^*D)\subseteq f^{-1}\mathcal J_{NLC}(X, D) 
$$ 
of ideal sheaves on $Y$. 
\end{cor}
\begin{proof}
See, for example, \cite[Example 9.5.8]{lazarsfeld}. 
\end{proof}

\begin{prop}[Divisors of small multiplicity]
Let $D$ be an effective 
$\mathbb R$-divisor on a smooth variety $X$. 
Suppose that $x\in X$ is a point at which 
$\mult _x D\leq 1$. 
Then the ideal $\mathcal J_{NLC}(X, D)$ is trivial at 
$x$. 
\end{prop}
\begin{proof}
It is obvious. See, for example, \cite[Proposition 9.5.13]{lazarsfeld}. 
\end{proof}

\begin{thm}[Generic Restriction]\label{gene}
Let $X$ and $T$ be smooth irreducible varieties, 
and $p:X\to T$ a smooth surjective morphism. 
Consider an effective $\mathbb R$-divisor $D$ on $X$ whose 
support does not contain any of the fibers 
$X_t=p^{-1}(t)$, so that 
for each $t\in T$ the restriction $D_t=D|_{X_t}$ is defined. 
Then there is a non-empty Zariski open set 
$U\subset T$ such that 
$$
\mathcal J_{NLC}(X_t, D_t)=\mathcal J_{NLC}(X, D)_t
$$ 
for every $t\in U$, where 
$\mathcal J_{NLC}(X, D)_t=\mathcal J_{NLC}(X, D)\cdot \mathcal O_{X_t}$ 
denotes the restriction of the indicated 
non-lc ideal to the fiber $X_t$. More generally, 
if $t\in U$ then 
$$\mathcal J_{NLC}(X_t, c\cdot D_t)=\mathcal J_{NLC}(X, 
c\cdot 
D)_t$$ 
for every $c>0$. 
\end{thm}
\begin{proof}
We use the same notation as in the proof of \cite[Theorem 9.5.35]{lazarsfeld}. 
Let $U$ be the non-empty Zariski open set of $T$ that was obtained 
in the proof of \cite[Theorem 9.5.35]{lazarsfeld}. 
By shrinking $T$, we can assume that 
$T=U$. 
We take a general hypersurface $H$ of $T$ 
passing through $t\in U$. 
Then $\mathcal J_{NLC}(X, c\cdot D)=\mathcal J_{NLC}(X, 
X_1+c\cdot D)$, 
where $X_1=p^*H$. 
By Theorem \ref{main}, 
\begin{align*}
\mathcal J_{NLC}(X, c\cdot D)|_{X_1}&=\mathcal 
J_{NLC}(X, X_1+c\cdot D)|_{X_1}\\ 
&=\mathcal J_{NLC}(X_1, c\cdot D|_{X_1}). 
\end{align*}
By applying this argument $\dim T$ times, 
we obtain that 
$\mathcal J_{NLC}(X_t, c\cdot D_t)=\mathcal J_{NLC}(X, 
c\cdot 
D)_t$. 
\end{proof}

The following corollary is a direct consequence 
of Theorem \ref{gene}. 

\begin{cor}[Semicontinuity]\label{215} 
Let $p:X\to T$ be a smooth morphism 
as in {\em{Theorem \ref{gene}}}, and let $D$ be an effective 
$\mathbb R$-divisor 
on $X$ satisfying the hypotheses of 
that statement. 
Suppose moreover given a section $y:T\to X$ of $p$, and 
write $y_t=y(t)\in X$. 
If $y_t\in \Nlc(X_t, D_t)$ for 
$t\ne 0\in T$, 
then $y_0\in \Nlc(X_0, D_0)$. 
\end{cor}
\begin{proof}
See the proof of \cite[Corollary 9.5.39]{lazarsfeld}. 
\end{proof}

We close this subsection with the subadditivity theorem 
for non-lc ideal sheaves (cf.~\cite{del}). 

\begin{thm}[Subadditivity] 
Let $X$ be a smooth 
variety. 
\begin{itemize}
\item[(1)] Suppose that $D_1$ and $D_2$ are any two effective 
$\mathbb R$-divisor on $X$. 
Then 
$$
\mathcal J_{NLC}(X, D_1+D_2)\subseteq \mathcal J_{NLC}(X, D_1)\cdot 
\mathcal J_{NLC}(X, D_2). 
$$
\item[(2)] If $\mathfrak a, \mathfrak b\subseteq \mathcal O_X$ are ideal 
sheaves, 
then 
$$
\mathcal J_{NLC}(X; \mathfrak a^c\cdot \mathfrak b^d)
\subseteq \mathcal J_{NLC}(X; \mathfrak a^c)\cdot 
\mathcal J_{NLC}(X; \mathfrak b^d) 
$$ 
for any $c, d>0$. 
In particular, 
$$
\mathcal J_{NLC}(X; \mathfrak a\cdot \mathfrak b)
\subseteq \mathcal J_{NLC}(X; \mathfrak a)\cdot 
\mathcal J_{NLC}(X; \mathfrak b).  
$$ 
\end{itemize}
\end{thm}
\begin{proof}
The proof of the subadditivity theorem for multiplier ideal sheaves 
works for non-lc ideal sheaves. See, for example, the proof 
of \cite[Theorem 
9.5.20]{lazarsfeld}. 
We leave the details as an exercise for the reader. 
\end{proof}

\section{Miscellaneous Results}\label{sec3}
In this section, we collect some basic results of non-lc ideal 
sheaves. 

\subsection{Vanishing and Global Generation Theorems} 

Here, we state vanishing and global generation theorems explicitly. 
We can easily check them as applications of Theorem \ref{ap1} below. 

\begin{thm}[Vanishing Theorem]\label{vani} 
Let $X$ be a smooth projective 
variety, let $D$ be any $\mathbb R$-divisor on $X$, 
and let $L$ be any integral divisor such that 
$L-D$ is ample. 
Then $$H^i(X, \mathcal O_X(K_X+L)\otimes 
\mathcal J_{NLC}(X, D))=0$$ for 
$i>0$. 
\end{thm}

\begin{proof}
Let $f:Y\to X$ be a resolution with 
$K_Y+B_Y=f^*(K_X+D)$ such that $\Supp B_Y$ is 
simple normal crossing. 
Then 
$$\ulcorner -(B^{<1}_Y)\urcorner -\llcorner B^{>1}_Y\lrcorner 
+f^*(K_X+L) -(K_Y+B^{=1}_Y+\{B_Y\})=f^*(L-D). 
$$ 
Therefore, $H^i(X, R^jf_*\mathcal O_Y
(\ulcorner -(B^{<1}_Y)\urcorner -\llcorner B^{>1}_Y\lrcorner 
+f^*(K_X+L)))=0$ for every $i>0$ and $j\geq 0$ by 
Theorem \ref{ap1} (2). 
In particular, 
$$H^i(X, f_*\mathcal O_Y
(\ulcorner -(B^{<1}_Y)\urcorner -\llcorner B^{>1}_Y\lrcorner 
+f^*(K_X+L)))=0$$ for $i>0$. 
This is the desired vanishing theorem 
because $\mathcal J_{NLC}(X, D)=f_*\mathcal O_Y
(\ulcorner -(B^{<1}_Y)\urcorner -\llcorner B^{>1}_Y\lrcorner)$. 
\end{proof}

We can weaken the assumption in Theorem \ref{vani}. 
However, Theorem \ref{vani} is sufficient for our purpose in this 
paper. So, the reader can skip the next difficult theorem. 

\begin{thm}\label{vani2} 
Let $X$ be a normal variety and $\Delta$ 
an effective $\mathbb R$-divisor 
such that $K_X+\Delta$ is $\mathbb R$-Cartier. 
Let $\pi:X\to V$ be a proper 
morphism onto an algebraic variety $V$ and 
$L$ a Cartier divisor on $X$. 
Assume that $L-(K_X+\Delta)$ is $\pi$-nef and $\pi$-log 
big with respect to $(X, \Delta)$, 
that is, $L-(K_X+\Delta)$ is $\pi$-nef and $\pi$-big and 
$(L-(K_X+\Delta))|_C$ is $\pi$-big for every lc center $C$ of the 
pair $(X, \Delta)$. 
Then we have 
$$
R^i\pi_*(\mathcal J_{NLC}(X, \Delta)\otimes \mathcal O_X(L))=0
$$ 
for all $i>0$. 
\end{thm}
\begin{proof}
Let $f:Y\to X$ be a resolution with 
$K_Y+\Delta_Y=f^*(K_X+\Delta)$ such that 
$\Supp \Delta_Y$ is simple normal crossing. 
We put $F=\Delta^{=1}_Y-E$, where 
$E$ is the union of irreducible components of $\Delta^{=1}_Y$ which are 
mapped to $X_{NLC}=\Nlc (X, \Delta)$. 
If we need, we take more blow-ups and 
can assume that no strata of $F$ are mapped 
to $X_{NLC}$. 
In this case, we have 
$$\mathcal J_{NLC}(X, \Delta)= 
f_*\mathcal O_Y(\ulcorner -(\Delta^{<1}_Y)\urcorner -\llcorner \Delta^{>1}_Y
\lrcorner-E). 
$$ 

Since 
\begin{align*}
&\ulcorner -(\Delta^{<1}_Y)\urcorner -\llcorner \Delta^{>1}_Y\lrcorner -E
+f^*L-(K_Y+F+\{\Delta_Y\})\\&=f^*(L-(K_X+\Delta)), 
\end{align*} 
we have that 
$$R^i\pi_*R^jf_*\mathcal O_Y
(\ulcorner -(\Delta^{<1}_Y)\urcorner -\llcorner \Delta^{>1}_Y\lrcorner-E 
+f^*L)=0$$ for every $i>0$ and $j\geq 0$ (see, for example, 
\cite[Theorem 2.47]{fujino-book}).
So, we obtain 
$$R^i\pi_*(\mathcal J_{NLC}(X, \Delta)\otimes \mathcal O_X(L))=0$$ for $i>0$. 
\end{proof}

\begin{thm}[Global Generation]\label{gg}  
Let $X$ be a smooth projective 
variety of dimension $n$. We fix a globally generated ample 
divisor $B$ on $X$. 
Let $D$ be an effective $\mathbb R$-divisor 
and $L$ an integral divisor on $X$ such that 
$L-D$ is ample $($or, more generally, nef and 
log big with respect to $(X, D)$$)$. 
Then $\mathcal O_X(K_X+L+mB)\otimes 
\mathcal J_{NLC}(X, D)$ is globally generated as 
soon as $m\geq n$. 
\end{thm}
\begin{proof}
It is obvious by Theorem \ref{vani} (or, Theorem \ref{vani2}) and 
Mumford's $m$-regularity. 
\end{proof}

\subsection{Asymptotic non-lc ideal sheaves}

Let $X$ be a smooth variety. 
Let $\mathfrak a_{\bullet}=\{\mathfrak a_m\}$ 
be a graded system of ideals on $X$. 
In other words, $\mathfrak a_{\bullet}$ consists of a collection of ideal 
sheaves $\mathfrak a_k\subseteq \mathcal O_X$ satisfying 
$\mathfrak a_0=\mathcal O_X$ and $\mathfrak a_m \cdot 
\mathfrak a_l \subseteq \mathfrak a_{m+l}$ for 
all $m, l \geq 1$. 

\begin{defn}[Non-lc ideal associated to a graded system 
of ideals] 
The {\em{asymptotic non-lc ideal sheaf}} 
of $\mathfrak a_{\bullet}$ with {\em{coefficient}} 
or {\em{exponent}} $c$, written 
either by 
$$
\mathcal J_{NLC}(X; c\cdot \mathfrak a_{\bullet}) \ \ 
\text{or} \ \ \mathcal J_{NLC}(X; \mathfrak a^c_{\bullet}) 
$$ 
is defined to be the unique 
maximal member among the 
family of ideals 
$\{\mathcal J_{NLC}(X; \frac{c}{p}\cdot \mathfrak a_p)\}$ for 
$p\geq 1$. Thus $\mathcal J_{NLC}(X; c\cdot \mathfrak 
a_{\bullet})=\mathcal J_{NLC}(X; \frac {c}{p}\cdot \mathfrak a_p)$ 
for all sufficiently large and divisible 
integer $p\gg 0$. 
\end{defn}
\begin{ex}
Let $X$ be a smooth projective variety and $L$ 
an integral divisor on $X$ of non-negative Iitaka 
dimension. 
We consider the base ideal $\mathfrak b_k =\mathfrak b (|kL|)$ of the 
complete linear system $|kL|$ for every $k\geq 0$. 
Let $\Delta$ be an effective $\mathbb R$-divisor on $X$ 
such that $K_X+\Delta$ is $\mathbb R$-Cartier. 
Then $\mathfrak b_{\bullet}$ is a graded system of 
ideals on $X$. We put $$
\mathcal J_{NLC}((X, \Delta), |\!|L|\!|):=
\mathcal J_{NLC}((X, \Delta); \mathfrak b_{\bullet}). 
$$ 
We note that $\mathcal J_{NLC}((X, \Delta); \mathfrak b_{\bullet})$ is the unique 
maximal member among the family of ideals $\{\mathcal J_{NLC}((X, \Delta); 
\frac{1}{p} \mathfrak b_p)\}$ for $p\geq 1$. 
\end{ex}

Almost all the basic properties 
of asymptotic multiplier ideal sheaves 
in \cite[11.1 and 11.2.A]{lazarsfeld} 
can be proved for asymptotic non-lc ideal sheaves 
by the same arguments. 
Therefore, we do not repeat them 
here. 
We leave them as exercises for the reader. 
We state only one theorem in this subsection. 

\begin{thm}
Let $X$ be a smooth projective variety, $\Delta$ an effective 
Cartier divisor on $X$, and $L$ an integral divisor 
on $X$ of non-negative 
Iitaka dimension. 
If $A$ is an ample divisor on $X$, 
then $$H^i(X, \mathcal O_X(K_X+\Delta+mL+A)\otimes 
\mathcal J_{NLC}((X, \Delta), |\!|mL|\!|))=0$$ 
for $i>0$. Furthermore, we assume that $B$ is 
a globally generated 
ample divisor 
on $X$. Then for every $m\geq 1$, 
$$\mathcal O_X(K_X+\Delta+lB+A+mL)\otimes 
\mathcal J_{NLC}((X, \Delta), |\!|mL|\!|)$$ is globally generated 
as soon as $l\geq \dim X$. 
\end{thm}
\begin{proof}
Let $H\in |kmL|$ be a general member for a large 
and divisible $k$. Then $\mathcal J_{NLC}((X, \Delta), 
|\!|mL|\!|)=\mathcal J_{NLC}((X, \Delta), \frac {1}{k}H)
=\mathcal J_{NLC}(X, \Delta +\frac{1}{k}H)$. 
On the other hand, $\Delta +mL+A-(\Delta +\frac{1}{k}H)\sim 
_{\mathbb Q}A$. 
Thus, 
this theorem follows from Theorem \ref{vani} and Theorem \ref{gg}. 
\end{proof}

\section{Appendix}\label{sec4}  
\subsection{Inversion of adjunction on log canonicity} 
We give some comments on the inversion of adjunction 
on log canonicity. 
The following theorem is due to Kawakita. 
Roughly speaking, he proved it by iterating the 
restriction theorem between {\em{adjoint ideal sheaves}} on $X$  
and {\em{multiplier ideal sheaves}} on $S$. 
For the proof, see \cite{kawakita}. 

\begin{thm}[Kawakita]\label{kawa-thm}
Let $X$ be a normal variety, $S$ a reduced divisor on $X$, and $B$ 
an effective $\mathbb R$-divisor 
on $X$ such that 
$K_X+S+B$ is $\mathbb R$-Cartier. 
Assume that $S$ has no common irreducible component with the 
support of $B$. 
Let $\nu: S^{\nu}\to S$ be the normalization and  
$B_{S^{\nu}}$ the different on $S^{\nu}$ such that 
$K_{S^{\nu}}+B_{S^{\nu}}=\nu^*((K_X+S+B)|_S)$. 
Then $(X, S+B)$ is log canonical 
around $S$ if and only if 
$(S^{\nu}, B_{S^{\nu}})$ is log canonical. 
\end{thm}

By adjunction, it is obvious that $(S^{\nu}, B_{S^{\nu}})$ is log canonical 
if $(X, S+B)$ is log canonical 
around $S$. So, the above theorem is usually called the 
{\em{inversion of adjunction on log canonicity}}. 
We need the following corollary of Theorem \ref{kawa-thm} 
in the proof of Theorem \ref{main}. 
The proof is obvious. 

\begin{cor}\label{renketsu}
Let $(X, S+B)$ be as in {\em{Theorem \ref{kawa-thm}}}. 
Let $P\in X$ be a closed point such that 
$(X, S+B)$ is not log canonical at $P$. 
Let $f:Y\to X$ be a resolution such that 
$K_Y+B_Y=f^*(K_X+S+B)$ and 
that $\Supp B_Y$ is simple normal crossing. 
Then $f^{-1}(P)\cap S_Y\cap \Supp N\ne \emptyset$, 
where $S_Y=f^{-1}_*S$ and $N=\llcorner B^{>1}_Y\lrcorner$. 
\end{cor}

We close this subsection with a remark on the theory of quasi-log varieties. 

\begin{rem}We use the notation in Theorem \ref{kawa-thm}. 
We note that 
$[X, K_X+S+B]$ has a natural quasi-log structure, which was introduced 
by Ambro. See, for example, \cite[Chapter 3]{fujino-book}. 
By adjunction, $S'=S\cup X_{NLC}$ has a natural quasi-log structure 
induced by $[X, K_X+S+B]$. 
More explicitly, the defining ideal sheaf of the quasi-log variety $S'$ 
is $J$ in the proof of Theorem \ref{main}. 
In Step \ref{step1} in the proof of 
Theorem \ref{main}, we did not use the normality of 
$S$. Theorem \ref{kawa-thm} says that 
$[S', (K_X+S+B)|_{S'}]$ has only qlc singularities 
around $S$ if and only if $(S^{\nu}, B_{S^{\nu}})$ is lc. 
\end{rem}

\subsection{New Cohomological Package}
We quickly review Ambro's formulation 
of torsion-free and vanishing theorems in a simplified form. 
For more advanced topics and the proof, see \cite[Chapter 2]{fujino-book}. 

Let $Y$ be a simple normal crossing divisor 
on a smooth 
variety $M$ and $D$ an $\mathbb R$-divisor 
on $M$ such that 
$\Supp (D+Y)$ is simple normal crossing and that 
$D$ and $Y$ have no common irreducible components. 
We put $B=D|_Y$ and consider the pair $(Y, B)$. 
Let $\nu:Y^{\nu}\to Y$ be the normalization. 
We put $K_{Y^\nu}+\Theta=\nu^*(K_Y+B)$. 
A {\em{stratum}} of $(Y, B)$ is an irreducible component of $Y$ or 
the image of some lc center of $(Y^\nu, \Theta^{=1})$. 

When $Y$ is smooth and $B$ is an $\mathbb R$-divisor 
on $Y$ such that 
$\Supp B$ is simple normal crossing, we 
put $M=Y\times \mathbb A^1$ and $D=B\times \mathbb A^1$. 
Then $(Y, B)\simeq (Y\times \{0\}, B\times \{0\})$ satisfies 
the above conditions. 

\begin{thm}\label{ap1}
Let $(Y, B)$ be as above. 
Assume that $B$ is a boundary $\mathbb R$-divisor. 
Let 
$f:Y\to X$ be a proper morphism and $L$ a Cartier 
divisor on $Y$. 

$(1)$ Assume that $H\sim _{\mathbb R}L-(K_Y+B)$ is $f$-semi-ample.
Then 
every non-zero local section of $R^qf_*\mathcal O_Y(L)$ contains 
in its support the $f$-image of 
some stratum of $(Y, B)$. 

$(2)$ Let $q$ be an arbitrary non-negative 
integer. Let $\pi:X\to V$ be a proper morphism and 
assume that $H\sim _{\mathbb R}f^*H'$ for 
some $\pi$-ample $\mathbb R$-Cartier 
$\mathbb R$-divisor $H'$ on $X$. 
Then, $R^qf_*\mathcal O_Y(L)$ is $\pi_*$-acyclic, that is, 
$R^p\pi_*R^qf_*\mathcal O_Y(L)=0$ for every $p>0$. 
\end{thm}

For the proof, see \cite[Theorem 2.39]{fujino-book}. 
We note that \cite{fuji-pro} is a gentle introduction to this new 
cohomological package. 
The reader can find various applications in \cite{fujino-eff}, 
\cite{fujino3}, \cite{fuji-quasi}, \cite{fujino-book}, 
\cite{fuji-non}, and \cite{fuji-f}. 

\ifx\undefined\bysame
\newcommand{\bysame|{leavemode\hbox to3em{\hrulefill}\,}
\fi

\end{document}